\documentclass[12pt,a4paper]{article}
\usepackage[utf8x]{inputenc}
\usepackage{amsmath}
\usepackage{amsfonts}
\usepackage{amssymb}
\usepackage{amsthm}
\usepackage{mathrsfs}
\usepackage{fancyhdr}
\usepackage{graphicx}
\usepackage[usenames,x11names]{xcolor}
\usepackage{arabtex}
\usepackage{framed}
\usepackage[top=2cm,bottom=2cm,margin=2cm]{geometry}
\usepackage{cancel}
\usepackage{array}   %%%% Pour les fonctionalités assez évoluées d'un tableau

\usepackage{hyperref}

\title{The integrality of the Genocchi numbers obtained through a new identity and other results}
\author{\sc Bakir FARHI \\
Laboratoire de Mathématiques appliquées \\
Faculté des Sciences Exactes \\
Université de Bejaia, 06000 Bejaia, Algeria \\[1mm]
\href{mailto:bakir.farhi@gmail.com}{bakir.farhi@gmail.com} \\[1mm]
\url{http://farhi.bakir.free.fr/}
}
\date{}

\def\R{{\mathbb R}}

\def\N{{\mathbb N}}
\def\Z{{\mathbb Z}}
\def\lcm{\mathrm{lcm}}

\def\EMdash{\leavevmode\hbox to 10.6mm{\vrule height .63ex depth -.59ex
    width 10mm\hfill}}

\theoremstyle{plain}
\numberwithin{equation}{section}
\newtheorem{thm}{Theorem}[section]

\newtheorem{prop}[thm]{Proposition}
\newtheorem{coll}[thm]{Corollary}

\newtheorem{propn}{Proposition} %%% Proposition numérotée sans suivre sa section

\theoremstyle{definition}

\theoremstyle{remark}

\begin{document}
\maketitle

\vspace*{-11cm}

\begin{minipage}{0.6\textwidth}
\textit{This note replaces and improves the previous one, entitled ``On a curious integer sequence'' and referenced by: \href{http://arxiv.org/abs/2204.10136}{arXiv:2204.10136v1 [math.NT] 21 Apr 2022}.}
\end{minipage}

\vspace*{9cm}

\begin{abstract}
In this note, we investigate some properties of the integer sequence of general term $a_n := \sum_{k = 0}^{n - 1} k! (n - k - 1)!$ ($\forall n \geq 1$) to derive a new identity of the Genocchi numbers $G_n$ ($n \in \N$), which immediately shows 
that $G_n \in \Z$ for any $n \in \N$. In another direction, we obtain nontrivial lower bounds for the $2$-adic valuations of the rational numbers $\sum_{k = 1}^{n} \frac{2^k}{k}$.      
\end{abstract}


\section{Introduction and Notation}\label{sec1}

Throughout this note, we let $\N^*$ denote the set of positive integers. For $x \in \R$, we let $\lfloor x\rfloor$ denote the integer part of $x$. For a given prime number $p$ and a given positive integer $n$, we let $\vartheta_p(n)$ and $s_p(n)$ respectively denote the usual $p$-adic valuation of $n$ and the sum of base-$p$ digits of $n$. A well-known formula of Legendre (see e.g., \cite[Theorem 2.6.4, page 77]{moll}) states that for any prime number $p$ and any positive integer $n$, we have
\begin{equation}\label{eq12}
\vartheta_p(n!) = \frac{n - s_p(n)}{p - 1} .
\end{equation}
Next, we let $s(n , k)$ and $S(n , k)$ (with $n , k \in \N$, $n \geq k$) respectively denote the Stirling numbers of the first and second kinds, which can be defined as the integer coefficients appearing in the polynomial identities:
\begin{equation*}
\begin{split}
X (X - 1) \cdots (X - n + 1) & = \sum_{k = 0}^{n} s(n , k) X^k , \\
X^n & = \sum_{k = 0}^{n} S(n , k) X (X - 1) \cdots (X - k + 1) 
\end{split} ~~~~~~~~~~ (\forall n \in \N) .
\end{equation*}
This immediately implies the orthogonality relations:
\begin{equation}\label{eq13}
\sum_{k \leq i \leq n} s(n , i) S(i , k) = \sum_{k \leq i \leq n} S(n , i) s(i , k) = \delta_{n k} ~~~~~~~~~~ (\forall n , k \in \N , n \geq k) ,
\end{equation}
where $\delta_{n k}$ is the Kronecker delta. Among the many formulas related to the Stirling numbers, we mention the following (see e.g., \cite[§1.14, page 51]{com}): 
\begin{equation}\label{eq5}
\frac{\log^k(1 + x)}{k!} = \sum_{n = k}^{+ \infty} s(n , k) \frac{x^n}{n!} ~~~~~~~~~~ (\forall k \in \N) ,
\end{equation}
which is needed later on. We let finally $B_n$ and $G_n$ ($n \in \N$) respectively denote the Bernoulli and the Genocchi numbers which can be defined by their respective exponential generating functions:
\begin{equation}\label{eq4}
\frac{x}{e^x - 1} = \sum_{n = 0}^{+ \infty} B_n \frac{x^n}{n!} ~~~~\text{and}~~~~ \frac{2 x}{e^x + 1} = \sum_{n = 0}^{+ \infty} G_n \frac{x^n}{n!} . 
\end{equation} 
The famous Genocchi theorem \cite{gen} states that the $G_n$'s are all integers. There are at least two beautiful proofs of the Genocchi theorem: the first one uses the formula $G_n = 2 (1 - 2^n) B_n$ (see e.g., \cite{com}) together with the Fermat little theorem and the von Staudt-Clausen theorem, while the second one uses the remarkable Seidel formula \cite{sei}:
$$
\sum_{k = 0}^{n} \binom{n}{k} G_{n + k} = 0 ~~~~~~~~~~ (\forall n \in \N) .
$$
In this note, we give a new proof of the integrality of the $G_n$'s by expressing them in terms of the Stirling numbers of the second kind. The starting point of this research is the study of the integer sequence ${(a_n)}_{n \in \N}$, defined by:
\begin{equation}\label{eqq1}
a_0 = 0 ~~\text{and}~~ a_n := \sum_{k = 0}^{n - 1} k! (n - k - 1)! ~~~~ (\forall n \in \N^*) .
\end{equation}
This sequence is closely related to the sum of the inverses of binomial coefficients which is studied by several authors (see \cite{jua,roc,sur1,sur2,tri,yan,zha}). It must be noted that both Stirling numbers, Genocchi numbers and the numbers $a_n$ ($n \in \N$) have combinatorial interpretations (see e.g., \cite{com,sta} for the Stirling numbers, \cite{dum,vie} for the Genocchi numbers and the sequence \href{https://oeis.org/search?q=A003149&language=english&go=Search}{A003149} of \cite{slo} for the $a_n$'s). 

In another direction, by leaning on Legendre's formula \eqref{eq12}, an identity due to Rockett \cite{roc}, and another identity due to the author \cite{far}, we obtain nontrivial lower bounds for the $2$-adic valuations of the rational numbers $\sum_{k = 1}^{n} \frac{2^k}{k}$ ($n \in \N^*$). 

\section{The results and the proofs}

Our main result is the following:

\begin{thm}\label{t1}
For all natural number $n$, we have
\begin{equation}\label{eqq2}
G_n = \sum_{1 \leq \ell \leq k \leq n} (-1)^{k - 1} (\ell - 1)! (k - \ell)! S(n , k) .
\end{equation}
In particular, $G_n$ is an integer for any $n \in \N$.
\end{thm}

To prove this theorem, we need some intermediary results. The first one (Proposotion \ref{p1} below) can be immediately derived from the following identity of Rockett \cite{roc}:
\begin{equation}\label{eqq3}
\sum_{k = 0}^{n} \frac{1}{\binom{n}{k}} = \frac{n + 1}{2^{n + 1}} \sum_{k = 1}^{n + 1} \frac{2^k}{k} ~~~~~~~~~~ (\forall n \in \N) .
\end{equation}
But for convenience, we prefer reproduce its proof here.

\begin{prop}\label{p1}
For all natural number $n$, we have
\begin{equation}\label{eq1}
a_n = \frac{n!}{2^n} \sum_{k = 1}^{n} \frac{2^k}{k} .
\end{equation}
\end{prop}

\begin{proof}
We begin by establishing a recurrent formula for the sequence ${(a_n)}_n$. For any integer $n \geq 2$, we have
\begin{align*}
a_n & := \sum_{k = 0}^{n - 1} k! (n - k - 1)! \\
& = \sum_{k = 0}^{n - 2} k! (n - k - 1)! + (n - 1)! \\
& = \sum_{k = 0}^{n - 2} k! (n - k - 2)! (n - k - 1) + (n - 1)! \\
& = n \sum_{k = 0}^{n - 2} k! (n - k - 2)! - \sum_{k = 0}^{n - 2} (k + 1)! (n - k - 2)! + (n - 1)! .
\end{align*}
But since
$$
\sum_{k = 0}^{n - 2} k! (n - k - 2)! = a_{n - 1}
$$
and
\begin{align*}
\sum_{k = 0}^{n - 2} (k + 1)! (n - k - 2)! & = \sum_{\ell = 1}^{n - 1} \ell! (n - \ell - 1)! ~~~~~~~~~~ (\text{by putting } \ell = k + 1) \\
& = \sum_{\ell = 0}^{n - 1} \ell! (n - \ell - 1)! - (n - 1)! \\
& = a_n - (n - 1)! ,
\end{align*}
it follows that:
$$
a_n = n a_{n - 1} - a_n + 2 \cdot (n - 1)! .
$$
Hence
\begin{equation}\label{eqq4}
a_n = \frac{n}{2} a_{n - 1} + (n - 1)! .
\end{equation}
Further, we remark that Formula \eqref{eqq4} also holds for $n = 1$. Now, according to Formula \eqref{eqq4}, we have for any positive integer $k$:
$$
\frac{2^k}{k!} a_k - \frac{2^{k - 1}}{(k - 1)!} a_{k - 1} = \frac{2^k}{k} .
$$
Then by summing both sides of the last equality from $k = 1$ to $n$, we obtain (because the sum on the left is telescopic and $a_0 = 0$) that:
$$
\frac{2^n}{n!} a_n = \sum_{k = 1}^{n} \frac{2^k}{k} ,
$$
which gives the required formula. The proof is achieved.
\end{proof}

\begin{coll}\label{coll1}
The exponential generating function of the sequence ${(a_n)}_n$ is given by:
\begin{equation}\label{eq2}
\sum_{n = 0}^{+ \infty} a_n \frac{x^n}{n!} = \frac{- 2 \log(1 - x)}{2 - x} .
\end{equation}
\end{coll}

\begin{proof}
Using Formula \eqref{eq1} of Proposition \ref{p1}, we have
$$
\sum_{n = 0}^{+ \infty} a_n \frac{x^n}{n!} = \sum_{n = 1}^{+ \infty} \left(\frac{1}{2^n} \sum_{k = 1}^{n} \frac{2^k}{k}\right) x^n = \sum_{k = 1}^{+ \infty} \sum_{n = k}^{+ \infty} \frac{1}{2^n} \frac{2^k}{k} x^n = \sum_{k = 1}^{+ \infty} \frac{2^k}{k} \left(\sum_{n = k}^{+ \infty} \left(\frac{x}{2}\right)^n\right) .
$$ 
But since $\sum_{n = k}^{+ \infty} \left(\frac{x}{2}\right)^n = \left(\frac{x}{2}\right)^k \frac{1}{1 - \frac{x}{2}} = \frac{x^k}{2^k} \cdot \frac{2}{2 - x}$, we get
$$
\sum_{n = 0}^{+ \infty} a_n \frac{x^n}{n!} = \frac{2}{2 - x} \sum_{k = 1}^{+ \infty} \frac{x^k}{k} = \frac{2}{2 - x} \left(- \log(1 - x)\right) , 
$$
as required. This achieves the proof.
\end{proof}

Next, from Corollary \ref{coll1} and Formula \eqref{eq5}, we derive the following corollary:

\begin{coll}\label{coll5}
For any natural number $n$, we have
\begin{equation}\label{eqq5}
a_n = (-1)^{n - 1} \sum_{k = 0}^{n} G_k s(n , k) .
\end{equation}
\end{coll}

\begin{proof}
Let us consider the following three functions (which are analytic on the neighborhood of zero):
$$
f(x) := \frac{- 2 \log(1 - x)}{2 - x} ~,~ g(x) := \frac{2 x}{e^x + 1} ~,~ \text{ and } h(x) := \log(1 - x) .
$$
We easily check that $f = - g \circ h$. Since in addition $h(0) = 0$ then the power series expansion of $f$ about the origin can be obtained by substituting $h$ in the power series expansion of $g$ about the origin (which is given by \eqref{eq4}) and multiplying by $(-1)$. Doing so, we get
\begin{equation}\label{eq3}
f(x) = - \sum_{k = 0}^{+ \infty} G_k \frac{(h(x))^k}{k!} = - \sum_{k = 0}^{+ \infty} G_k \frac{\log^k(1 - x)}{k!} . 
\end{equation}
Further, by substituting in \eqref{eq5} $x$ by $(- x)$, we have for any $k \in \N$:
$$
\frac{\log^k(1 - x)}{k!} = \sum_{n = k}^{+ \infty} (-1)^n s(n , k) \frac{x^n}{n!} .
$$
So, by inserting this last into \eqref{eq3}, we get
\begin{align*}
f(x) & = - \sum_{k = 0}^{+ \infty} G_k \sum_{n = k}^{+ \infty} (-1)^n s(n , k) \frac{x^n}{n!} \\
& = - \sum_{n = 0}^{+ \infty} \sum_{k = 0}^{n} (-1)^n G_k s(n , k) \frac{x^n}{n!} \\
& = \sum_{n = 0}^{+ \infty} \left[(-1)^{n - 1} \sum_{k = 0}^{n} G_k s(n , k)\right] \frac{x^n}{n!} .
\end{align*}
Comparing this with Formula \eqref{eq2} of Corollary \ref{coll1}, we conclude that:
$$
a_n = (-1)^{n - 1} \sum_{k = 0}^{n} G_k s(n , k) ~~~~~~~~~~ (\forall n \in \N) ,
$$
as required.
\end{proof}

We finally derive our main result from Corollary \ref{coll5} above by applying the well-known inversion formula recalled in the following proposition:

\begin{propn}\label{p2}
Let ${(u_n)}_{n \in \N}$ and ${(v_n)}_{n \in \N}$ be two real sequences. Then the two following identities $(I)$ and $(II)$ are equivalent:
\begin{align}
u_n & = \sum_{k = 0}^{n} v_k s(n , k) ~~~~~~~~~~ (\forall n \in \N) , \tag{$I$} \\
v_n & = \sum_{k = 0}^{n} u_k S(n , k) ~~~~~~~~~~ (\forall n \in \N) . \tag{$II$}
\end{align}
\end{propn}

\begin{proof}
Use the orthogonality relations \eqref{eq13} (see e.g., \cite{com} or \cite{rio} for the details).
\end{proof}

\begin{proof}[Proof of Theorem \ref{t1}]
It suffices to apply Proposition \ref{p2} for $u_n = (-1)^{n - 1} a_n$ and $v_n = G_n$ ($\forall n \in \N$). In view of \eqref{eqq5}, Identity $(I)$ holds; so $(II)$ also, that is
$$
G_n = \sum_{k = 0}^{n} (-1)^{k - 1} a_k S(n , k) ~~~~~~~~~~ (\forall n \in \N) .
$$
Finally, by substituting in this last equality $a_k$ by its expression given by \eqref{eqq1}, we get for any $n \in \N$:
\begin{align*}
G_n & = \sum_{k = 1}^{n} (-1)^{k - 1} \left(\sum_{i = 0}^{k - 1} i! (k - i - 1)!\right) S(n , k) \\
& = \sum_{k = 1}^{n} (-1)^{k - 1} \left(\sum_{\ell = 1}^{k} (\ell - 1)! (k - \ell)!\right) S(n , k) ~~~~~~~~~~ (\text{by putting } \ell = i + 1) \\
& = \sum_{1 \leq \ell \leq k \leq n} (-1)^{k - 1} (\ell - 1)! (k - \ell)! S(n , k) ,
\end{align*}
as required.
\end{proof}

Now, we turn to present another type of results providing nontrivial lower bounds for the $2$-adic valuations of the rational numbers $\sum_{k = 1}^{n} \frac{2^k}{k}$ ($n \in \N^*$).

\begin{thm}\label{t2}
For any positive integer $n$, we have
\begin{equation}\label{eq6}
\vartheta_2\left(\sum_{k = 1}^{n} \frac{2^k}{k}\right) \geq s_2(n) 
\end{equation}
and {\rm(}more strongly{\rm)}:
\begin{equation}\label{eqq6}
\vartheta_2\left(\sum_{k = 1}^{n} \frac{2^k}{k}\right) \geq n - \left\lfloor\frac{\log{n}}{\log{2}}\right\rfloor .
\end{equation}
\end{thm}  

\begin{proof}
Let $n$ be a fixed positive integer. Since $a_n \in \Z$ then we have $\vartheta_2(a_n) \geq 0$. But, by using Formula \eqref{eq1} of Proposition \ref{p1}, this is equivalent to:
$$
\vartheta_2(n!) - n + \vartheta_2\left(\sum_{k = 1}^{n} \frac{2^k}{k}\right) \geq 0 .
$$
Then, using the Legendre formula \eqref{eq12} for the prime number $p = 2$, which says that $\vartheta_2(n!) = \frac{n - s_2(n)}{2 - 1} = n - s_2(n)$, we get
$$
\vartheta_2\left(\sum_{k = 1}^{n} \frac{2^k}{k}\right) \geq s_2(n) ,
$$
confirming \eqref{eq6}. To establish the stronger lower bound \eqref{eqq6}, we use the Rockett formula \eqref{eqq3} together with the identity:
\begin{equation}\label{eqq7}
\lcm\left\{\binom{m}{0} , \binom{m}{1} , \dots , \binom{m}{m}\right\} = \frac{\lcm(1 , 2 , \dots , m , m + 1)}{m + 1} ~~~~~~~~~~ (\forall m \in \N) ,
\end{equation}
established by the author in \cite{far}. According to \eqref{eqq3} and \eqref{eqq7}, we have that
$$
\sum_{k = 1}^{n} \frac{2^k}{k} = \frac{2^n}{n} \sum_{k = 0}^{n - 1} \frac{1}{\binom{n - 1}{k}} ~~\text{and}~~ 1 = \frac{n}{\lcm(1 , 2 , \dots , n)} \cdot \lcm\left\{\binom{n - 1}{0} , \binom{n - 1}{1} , \dots , \binom{n - 1}{n - 1}\right\} .
$$
By multiplying side by side these last equalities, we get
$$
\sum_{k = 1}^{n} \frac{2^k}{k} = \frac{2^n}{\lcm(1 , 2 , \dots , n)} \cdot \lcm\left\{\binom{n - 1}{0} , \binom{n - 1}{1} , \dots , \binom{n - 1}{n - 1}\right\} \sum_{k = 0}^{n - 1} \frac{1}{\binom{n - 1}{k}} .
$$
But since the rational number $\lcm\{\binom{n - 1}{0} , \binom{n - 1}{1} , \dots , \binom{n - 1}{n - 1}\} \sum_{k = 0}^{n - 1} \frac{1}{\binom{n - 1}{k}}$ is obviously an integer, then we derive that:
$$
\vartheta_2\left(\sum_{k = 1}^{n} \frac{2^k}{k}\right) \geq \vartheta_2\left(\frac{2^n}{\lcm(1 , 2 , \dots , n)}\right) = n - \left\lfloor\frac{\log n}{\log 2}\right\rfloor ,
$$
confirming \eqref{eqq6} and completes the proof.
\end{proof}

\section{Two open problems}

\begin{enumerate}
\item Find a generalization of Theorem \ref{t2} to other prime numbers $p$ other than $p = 2$. Notice that the generalization that might immediately come to mind:
$$
\vartheta_p\left(\sum_{k = 1}^{n} \frac{p^k}{k}\right) \geq s_p(n) 
$$
is false for $p > 2$ (take for example $n = 2$).
\item Find a combinatorial interpretation for Formula \eqref{eqq2} of Theorem \ref{t1}. 
\end{enumerate}

\rhead{\textcolor{OrangeRed3}{\it References}}

\end{document}